\documentclass[11pt,leqno]{article}
\usepackage{amsmath, amscd, amsthm, amssymb, graphics, xypic, mathrsfs, setspace, fancyhdr, times, bm}
\DeclareFontFamily{OT1}{pzc}{}
\DeclareFontShape{OT1}{pzc}{m}{it}%
             {<-> s * [1.195] pzcmi7t}{}
\DeclareMathAlphabet{\mathscr}{OT1}{pzc}%
                                 {m}{it}
\usepackage[colorlinks=true,pagebackref=true]{hyperref} 
\hypersetup{backref}

\def\Z{\mathbb{Z}}
\def\N{\mathbb{N}}
\def\A{\mathbb{A}}
\def\OO{\mathcal{O}}

\def\Spec{\mathrm{Spec}}
\def\Char{\mathrm{char}}

\def\supp#1{\mathrm{supp}({#1})}

\newcommand{\sus}[1]{\mathrm{C}_*^{\mathrm{sing}}(#1)}

\newcommand{\aone}{\A^1}
\def\gm{\mathbb{G}_m} 
\def\P1{\mathbb{P}^{1}}

\def\Hom{\mathrm{Hom}}
\newcommand{\derR}{\mathbf{R}}
\DeclareMathOperator{\uHom}{\underline{Hom}} 

\def\K{\mathbf{K}}
\def\GW{\mathbf{GW}}
\def\MW{\mathrm{MW}}
\def\KMW{{\K}^{\mathrm{M\hspace{-.2ex}W}}} 

\def\H{\mathrm{H}}

\def\tc{\tilde{\mathrm{c}}}
\def\tZ{\tilde\Z}

\def\tcor#1{\widetilde{\mathrm{Cor}}_{#1}}

\def\Zar{\mathrm{Zar}}
\def\Nis{\mathrm{Nis}}

\def\Sm{\mathrm{Sm}}

\newcommand{\DMt}{\widetilde{\mathrm{DM}}^{\mathrm{eff}}(k,\Z)}
\newcommand{\DM}{\widetilde{\mathrm{DM}}(k,\Z)}

\newcommand{\tildediv}{\widetilde{\mathrm{div}}}

\newcounter{intro}
\theoremstyle{plain}
\newtheorem{thm}{Theorem}[subsection]

\newtheorem{lem}[thm]{Lemma}
\newtheorem{cor}[thm]{Corollary}
\newtheorem{prop}[thm]{Proposition}
\newtheorem*{claim*}{Claim}  

\newtheorem*{thm*}{Theorem}
\newtheorem*{problem*}{Problem}

\theoremstyle{definition}
\newtheorem{defn}[thm]{Definition}

\theoremstyle{remark}
\newtheorem{rem}[thm]{Remark}

\numberwithin{equation}{section}

\def\sAb{\mathbf{sAb}}
\def\tm{\widetilde{\mathcal{M}}^{\text{tr}}}
\def\m{\mathcal{M}}
\def\f{\mathbf{f}}
\def\tZ{\widetilde{\mathbb{Z}}}
\def\SH{\mathbf{SH}}
\def\KQ{\mathbf{KQ}}
\def\1{\mathbf{1}}

\begin{document}
\pagestyle{fancy}
\renewcommand{\sectionmark}[1]{\markright{\thesection\ #1}}
\fancyhead{}
\fancyhead[LO,R]{\bfseries\footnotesize\thepage}
\fancyhead[LE]{\bfseries\footnotesize\rightmark}
\fancyhead[RO]{\bfseries\footnotesize\rightmark}
\chead[]{}
\cfoot[]{}
\setlength{\headheight}{1cm}

\author{
Jean Fasel \and 
Paul Arne {\O}stv{\ae}r 
}

\title{{\bf A cancellation theorem for Milnor-Witt correspondences}}
\date{}
\maketitle

\begin{abstract}
We show that finite Milnor-Witt correspondences satisfy a cancellation theorem with respect to the pointed multiplicative group scheme.
This has several notable applications in the theory of Milnor-Witt motives and Milnor-Witt motivic cohomology.
\end{abstract}

\begin{footnotesize}
\setcounter{tocdepth}{1}
\tableofcontents
\end{footnotesize}

\section{Introduction}
The notion of finite Milnor-Witt correspondences was introduced in \cite{Calmes14b} and \cite{Deglise16} as a natural generalization of finite correspondences in the sense \cite{VSF}.
Intuitively one may view a Milnor-Witt correspondence as an ordinary correspondence together with a well behaved quadratic form over the function field of each irreducible component of the support of the correspondence.
This point of view allows to translate many results in the classical setting into this new framework.
In this paper we show the following result, Theorem \ref{thm:main} below, which is a cancellation theorem for Milnor-Witt correspondences. 

\begin{thm*}
Let $k$ be an infinite perfect field. 
Then for all complexes $\mathscr C$ and $\mathscr D$ of Milnor-Witt sheaves, 
the morphism
\begin{equation}
\label{cancellationisomorphism}
\Hom_{\DMt}(\mathscr C,\mathscr D)\to \Hom_{\DMt}(\mathscr C(1),\mathscr D(1))
\end{equation}
obtained by tensoring with the Tate object $\tZ(1)$ is an isomorphism.
\end{thm*}

Here, $\DMt$ is the analogue in this new framework of Voevodsky's category of motives $\mathrm{DM}^{\mathrm{eff}}(k)$ (see \cite{Deglise16}). 
This result extends Voevodsky's pioneering work on the cancellation theorem for finite correspondences in \cite{Voevodsky10}.
It basically asserts a suspension isomorphism in Milnor-Witt motivic cohomology with respect to the multiplicative group scheme $\gm$ viewed as an algebro-geometric sphere;
the Tate object $\tZ(1)$ is defined as $\tZ(\P1)/\tZ(\infty)[-2]$ in \cite[\S3.2]{Deglise16}.

\subsection*{Notation}
Throughout the paper we work over a perfect field $k$ of $\Char(k)\neq 2$. 
We denote by $\Sm_{k}$ the category of smooth separated schemes over $\Spec(k)$.
For a natural number $i\in \N$, an integer $j\in \Z$, and a line bundle $\mathcal L$ on $X\in\Sm_{k}$, 
we consider exterior derivatives of tangent bundles and set  
\[
C_i(X,\K_{j}^{\MW},\mathcal L):=\bigoplus_{x\in X^{(i)}} \K_{j-i}^{\MW}(k(x),\wedge^i(\mathfrak m_x/\mathfrak m_x^2)^\vee\otimes_{k(x)} \mathcal L_x).
\]
Here, 
$\K_{\ast}^{\MW}$ is the Nisnevich sheaf of unramified Milnor-Witt $K$-theory as defined in \cite[\S3]{Morel08}, 
and $\mathcal L_x$ is the stalk of $\mathcal L$ at a point $x\in X^{(i)}$ of codimension $i$.
We note that $\wedge^i(\mathfrak m_x/\mathfrak m_x^2)^\vee$ is a $1$-dimensional $k(x)$-vector space for every such $x$.
By means of the $\Z[\gm]$-module structures on $\K_{\ast}^{\MW}$ and $\Z[\mathcal L^{\times}]$, 
the Nisnevich sheaf of Milnor-Witt $K$-theory twisted with respect to $\mathcal L$ is defined in \cite[\S1.2]{Calmes14b} via the sheaf tensor product 
$$
\K_{\ast}^{\MW}(\mathcal L):=\K_{\ast}^{\MW}\otimes_{\Z[\gm]}\Z[\mathcal L^{\times}].
$$
The choice of a twisting line bundle on $X$ can be viewed as an "$\aone$-local system" on $X$.
The terms $C_i(X,\K_{j}^{\MW},\mathcal L)$ together with the differentials defined componentwise in \cite[\S5]{Morel08} form the Rost-Schmid cochain complex;
we let $\H^i(X,\KMW_j,\mathcal L)$ denote the associated cohomology groups.
We define the support of $\alpha\in C_i(X,\KMW_j,\mathcal L)$ to be the closure of the set of points $x\in X^{(i)}$ of codimension $i$ for which the $x$-component of $\alpha$ is nonzero.

To establish notation, 
we recall that the Milnor-Witt $K$-theory $\K_{\ast}^{\MW}(F)$ of any field $F$ is the quotient of the free associative algebra on generators $[F^{\times}]\cup\{\eta\}$ subject to the relations
\begin{itemize}
\item[(1)]
$[a][b]=0$, \text{for} $a+b=1$ (Steinberg relation),
\item[(2)]
$[a]\eta=\eta[a]$ ($\eta$-commutativity relation),
\item[(3)]
$[ab]=[a]+[b]+\eta[a][b]$ (twisted $\eta$-logarithm relation), and 
\item[(4)]
$(2+[-1]\eta)\eta=0$ (hyperbolic relation).
\end{itemize}
The grading on $\K_{\ast}^{\MW}(F)$ is defined by declaring $\vert \eta\vert=-1$ and $\vert [a]\vert =1$ for all $a\in F^{\times}$. 
For the Grothendieck-Witt ring of symmetric bilinear forms over $F$ there is a ring isomorphism 
$$
\GW(F)\xrightarrow{\simeq}\K_{0}^{\MW}(F);
\langle a\rangle\mapsto 1+\eta[a].
$$
We shall repeatedly use the zeroth motivic Hopf element $\epsilon:=-\langle -1\rangle=-1-\eta[-1]\in \K_{0}^{\MW}(F)$ in the context of $\aone$-homotopies between Milnor-Witt sheaves.
In fact, 
$\K_{\ast}^{\MW}(F)$ is isomorphic to the graded ring of endomorphisms of the motivic sphere spectrum over $F$ \cite[Theorem 6.2.1]{Morel04b}.

As in \cite[\S1]{Deglise16}, we denote by $\tc(X)$ the Milnor-Witt presheaf on $\Sm_{k}$ defined by 
$$
Y\mapsto \tc(X)(Y):=\tcor k(Y,X),
$$ 
and by $\tZ(X)$ its associated Milnor-Witt (Nisnevich) sheaf. 
One notable difference from the setting of finite correspondences is that the Zariski sheaf $\tc(X)$ is not in general a sheaf in the Nisnevich topology,
see \cite[Example 5.12]{Calmes14b}.
We write $\tc \{1\}$ for the cokernel of the morphism of presheaves 
$$
\tc (\Spec (k))\to \tc(\gm)
$$ 
induced by the $k$-rational point $1\colon\Spec(k)\to\gm$.
This is a direct factor of $\tc(\gm)$. 
As it turns out the associated Milnor-Witt sheaf $\tZ\{1\}$ is the cokernel of the morphism $\tZ (\Spec(k))\to \tZ(\gm)$ \cite[Proposition 1.2.11(2)]{Deglise16}. 

The additive category $\tcor k$ has the same objects as $\Sm_{k}$ and with morphisms given by finite Milnor-Witt correspondences \cite[Definition 1.1.5]{Deglise16}.
The 
cartesian product in $\Sm_{k}$ and the external product of finite Milnor-Witt correspondences
furnish a symmetric monoidal structure on $\tcor k$.
Taking graphs of maps in $\Sm_{k}$ yields a faithful symmetric monoidal functor 
$$
\tilde\gamma:\Sm_{k}\to\tcor k.
$$
The tensor product $\otimes$ on Milnor-Witt presheaves defined in \cite[\S1.2.13]{Deglise16} forms part of a closed symmetric monoidal structure;
the internal Hom object $\uHom$ is characterized by the property that $\uHom(\tc(X),{\mathcal F})(Y)={\mathcal F}(X\times Y)$ for any Milnor-Witt presheaf $\mathcal F$. 
The same holds for the category of Milnor-Witt sheaves;
here, 
$\uHom(\tZ(X),{\mathcal F})(Y)={\mathcal F}(X\times Y)$ for any Milnor-Witt sheaf ${\mathcal F}$.

\subsection*{Outline of the paper}

In \S\ref{section:cdMKK} we associate to a Cartier divisor $D$ on $X$ a class $\tildediv(D)$ in the cohomology group with support $\H^1_{\vert D\vert}(X,\K_1^{\MW},\OO_X(D))$.
One of our main calculations relates the push-forward of a principal Cartier divisor on $\gm$ defined by an explicit rational function to the zeroth motivic Hopf element $\epsilon\in\K_{0}^{\MW}$.
The cancellation theorem for Milnor-Witt correspondences is shown in \S\ref{cancellationW}.
In outline the proof follows the strategy in \cite{Voevodsky10}, 
but it takes into account the extra structure furnished by Milnor-Witt $K$-theory \cite{Morel08}.
A key result states that the twist map on the Milnor-Witt presheaf $\tc\{1\}\otimes\tc\{1\}$ is $\aone$-homotopic to $\epsilon$. 
An appealing aspect of the proof is that all the calculations with rational functions, Cartier divisors, and the residue homomorphisms in the Rost-Schmid cochain complex corresponding to points 
of codimension one or valuations can be carried out explicitly.

The cancellation theorem has several consequences.
In \S\ref{etfMWmotives} we show that tensoring complexes of Milnor-Witt sheaves $\mathscr C$ and $\mathscr D$ with $\tZ\{1\}$ induces the isomorphism (\ref{cancellationisomorphism}) for the 
category of effective Milnor-Witt motives.
This is the quadratic form analogue of Voevodsky's embedding theorem for motives.
In order to prove this isomorphism we first make a comparison of  Zariski and Nisnevich hypercohomology groups in \S\ref{ZvsN}.
Several other applications for Milnor-Witt motivic cohomology groups and Milnor-Witt motives follow in \cite{Deglise16}.

In \S\ref{examples} we apply the cancellation theorem to concrete examples which are used in \cite{Calmes14b} and \cite{Deglise16} for verifying the hyperbolic and $\eta$-twisted logarithmic relations 
appearing in the proof of the isomorphism between Milnor-Witt $K$-theory and the diagonal part of the Milnor-Witt motivic cohomology ring.

Finally, 
in \S\ref{MWmcs} we study Milnor-Witt motivic cohomology as a highly structured ring spectrum.  
The cancellation theorem shows this is an $\Omega_{T}$-spectrum.
This part is less detailed since it ventures into open problems on effective slices in the sense of \cite[\S5]{SO}.

\subsection*{Acknowledgments}
The first named author wishes to thank Fr\'ed\'eric D\'eglise for many useful conversations around the proof of the main theorem.
The second author was supported by the RCN Frontier Research Group Project no. 250399.
We gratefully acknowledge the hospitality and support of Institut Mittag-Leffler during Spring 2017.
Finally we thank H{\aa}kon Andreas Kolderup for his interest and careful reading of this paper.


\section{Cartier divisors and Milnor-Witt $K$-theory}
\label{section:cdMKK}

\subsection{Cartier divisors}
We refer to \cite[Appendix B.4]{Fulton98} for background on Cartier divisors.
Suppose $X$ is a smooth integral scheme and let $D=\{(U_i,f_i)\}$ be a Cartier divisor on $X$. 
We denote the support of $D$ by $\vert D\vert$.
The line bundle $\OO_X(D)$ associated to $D$ is the $\OO_X$-subsheaf of $k(X)$ generated by $f^{-1}_i$ on $U_i$ \cite[Appendix B.4.4]{Fulton98}. 
Multiplication by $f_i^{-1}$ yields an isomorphism between $\OO_{U_i}$ and the restriction of the line bundle $\OO_X(D)$ to $U_i$.

In the following we shall associate to the Cartier divisor $D$ a certain cohomology class 
$$
\tildediv(D)\in \H^1_{\vert D\vert}(X,\KMW_1,\OO_X(D)).
$$ 
Let $x\in X^{(1)}$ and choose $i$ such that $x\in U_i$. 
Then $f_i^{-1}$ is a generator of the stalk of $\OO_X(D)$ at $x$ and a fortiori a generator of $\OO_X(D)\otimes k(X)$. 
We may therefore consider the element 
$$
[f_i]\otimes (f_i)^{-1}\in \KMW_1(k(X),\OO_X(D)\otimes k(X)),
$$ 
and its boundary under the residue homomorphism
\[
\partial_x
\colon 
\KMW_1(k(X),\OO_X(D)\otimes k(X))
\to 
\KMW_0(k(x),(\mathfrak m_x/\mathfrak m_x^2)^\vee\otimes_{k(x)}\OO_X(D)_x).
\]
defined in \cite[Theorem 3.15]{Morel08} by keeping track of local orientations, see \cite[Remark 3.21]{Morel08}.

\begin{defn}
\label{def:Cartier}
With the notation above we set 
$$
\widetilde{\mathrm{ord}}_x(D):=\partial_x([f_i]\otimes (f_i)^{-1})
\in \KMW_0(k(x),(\mathfrak m_x/\mathfrak m_x^2)^\vee\otimes_{k(x)}\OO_X(D)_x),
$$ 
and 
\[
\widetilde{\mathrm{ord}}(D):=\sum_{x\in X^{(1)}\cap \vert D\vert} \widetilde{\mathrm{ord}}_x(D)\in C_1(X,\KMW_1,\OO_X(D)).
\]
\end{defn}

A priori it is not clear whether $\widetilde{\mathrm{ord}}_x(D)$ is well-defined for any $x\in X^{(1)}\cap \vert D\vert$ since the definition seems to depend on the choice of $U_i$. 
We address this issue in the following result.
\begin{lem}
\label{lem:Cartier}
Let $x\in X^{(1)}$ be such that $x\in U_i\cap U_j$. 
Then we have 
$$
\partial_x([f_i]\otimes (f_i)^{-1})=\partial_x([f_j]\otimes (f_j)^{-1}).
$$
\end{lem}
\begin{proof}
Note that $f_i/f_j\in \OO_{X}(U_i\cap U_j)^\times$ and a fortiori $f_i/f_j\in \OO_{X,x}^\times$. 
This implies the equalities
\[
[f_i]\otimes (f_i)^{-1}=[f_i]\otimes (f_j/f_i) (f_j)^{-1}=\langle f_j/f_i\rangle [f_i]\otimes (f_j)^{-1}.
\]
Using that $[f_j]=[(f_j/f_i)f_i]=[f_j/f_i]+\langle f_j/f_i\rangle [f_i]$ we deduce 
\[
[f_i]\otimes (f_i)^{-1}=[f_j]\otimes (f_j)^{-1}-[f_j/f_i]\otimes (f_j)^{-1}.
\]
Since $f_j/f_i\in \OO_{X,x}^\times$ we have $\partial_x([f_j/f_i])=0$, and the result follows.
\end{proof}

\begin{lem}
The boundary homomorphism
\[
d_1:C_1(X,\KMW_1,\OO_X(D))\to C_2(X,\KMW_1,\OO_X(D))
\]
vanishes on the class $\widetilde{\mathrm{ord}}(D)$ in Definition \ref{def:Cartier}, 
i.e., 
we have $d_1(\widetilde{\mathrm{ord}}(D))=0$.
\end{lem}
\begin{proof}
For $x\in X^{(2)}$ we set $Y:=\Spec(\OO_{X,x})$ and consider the boundary map
\[
d_{1,Y}:C_1(Y,\KMW_1,\OO_X(D)_{\vert_Y})
\to 
\KMW_{-1}(k(x),\wedge^2(\mathfrak m_x/\mathfrak m_x^2)^\vee\otimes_{k(x)}\OO_X(D)_x).
\]
Since the boundary homomorphism $d_1$ is defined componentwise it suffices to show the restriction $\widetilde{\mathrm{ord}}(D)_Y$ of $\widetilde{\mathrm{ord}}(D)$ to $Y$ vanishes under $d_{1,Y}$.
We may choose $i$ such that $x\in U_i$ and write
$$
\widetilde{\mathrm{ord}}(D)_Y=\sum_{y\in Y^{(1)}\cap \vert D\vert} \widetilde{\mathrm{ord}}_y(D).
$$ 
Then the map $Y\to U_i$ induces a commutative diagram of complexes:
\[
\xymatrix@C=1.6em{C_0(U_i,\K_1^{\MW},\OO_X(D)_{\vert_{U_i}})\ar[r]^-{d_{0,U_i}}\ar@{=}[d] & C_1(U_i,\K_1^{\MW},\OO_X(D)_{\vert_{U_i}})\ar[r]^-{d_{1,U_i}}\ar[d] & C_2(U_i,\K_1^{\MW},\OO_X(D)_{\vert_{U_i}})\ar[d] \\
C_0(Y,\K_1^{\MW},\OO_X(D)_{\vert_Y})\ar[r]^{d_{0,Y}} & C_1(Y,\K_1^{\MW},\OO_X(D)_{\vert_Y})\ar[r]^-{d_{1,Y}} & C_2(Y,\K_1^{\MW},\OO_X(D)_{\vert_Y})}
\]
Here $\widetilde{\mathrm{ord}}(D)_Y$ is the image of $\widetilde{\mathrm{ord}}(D)_{U_i}$ under the middle vertical map.
Thus it suffices to show that $d_{1,U_i}$ vanishes on $\widetilde{\mathrm{ord}}(D)_{U_i}$. 
The latter follows from the equality
$$
\widetilde{\mathrm{ord}}(D)_{U_i}=d_{0,U_i}([f_i]\otimes (f_i)^{-1}).
$$
\end{proof}

\begin{defn}
We write $\tildediv(D)$ for the class of $\widetilde{\mathrm{ord}}(D)$ in $\H^1_{\vert D\vert}(X,\K_1^{\MW},\OO_X(D))$. 
\end{defn}

\begin{rem}
The image of $\tildediv(D)$ under the extension of support homomorphism
\[
\H^1_{\vert D\vert}(X,\K_1^{\MW},\OO_X(D))\to \H^1(X,\K_1^{\MW},\OO_X(D))
\]
equals the Euler class of $\OO_X(D)^\vee$ \cite[\S 3]{Asok13}.
\end{rem}

Having defined the cocycle associated to a Cartier divisor, we can now intersect with cocycles. Let us first recall what we mean by a proper intersection.

\begin{defn}
Let $T\subset X$ be a closed subset of codimension $d$. If $D$ is a Cartier divisor on $X$, 
we say that $D$ and $T$ intersect properly if the intersection of $T$ and $D$ is proper, 
i.e., 
if any component of $\vert D\vert \cap T$ is of codimension $\geq d+1$. 
Equivalently, 
the intersection between $D$ and $T$ is proper if the generic points of $T$ are not in $\vert D\vert$.
\end{defn}

\begin{defn}
\label{defn:intersection}
Let $\mathcal L$ be a line bundle on $X$ and $\alpha\in H_T^i(X,\K_j^{\MW},\mathcal L)$. 
We write $D\cdot \alpha$ for the class of $\tildediv(D)\cdot \alpha$ in $H_{T\cap \vert D\vert}^{i+1}(X,\K_{j+1}^{\MW},\OO_X(D)\otimes \mathcal L)$ (and $\alpha\cdot D$ for the class of $\alpha\cdot \tildediv(D)$). Note in particular that $D\cdot 1=\tildediv(D)\in \H^1_{\vert D\vert}(X,\K_1^{\MW},\OO_X(D))$.
\end{defn}

\begin{rem}
\label{rem:commutativity}
It follows from \cite[Lemmas 4.19 and 4.20]{Fasel07} or \cite{Calmes17} that 
\[
D\cdot \alpha=\langle (-1)^j\rangle \alpha\cdot D.
\]
\end{rem}

Recall the canonical bundle of $X\in\Sm_{k}$ is defined by $\omega_{X/k}:= \wedge^{d_{X}}\Omega^{1}_{X/k}$.
Here, 
$\wedge^{d_{X}}$ is the exterior power to the dimension $d_{X}$ of $X$ and $\Omega^{1}_{X/k}$ is the cotangent bundle on $X$.

\begin{lem}
\label{lem:degree}
Let $D$ be the principal Cartier divisor on $\gm=\Spec (k[t^{\pm 1}])$ defined by 
$$
g:=(t^{n+1}-1)/(t^{n+1}-t)\in k(t).
$$ 
Denote by $\OO_{\gm}\to \OO_{\gm}(D)$ and $\OO_{\gm}\to \omega_{\gm/k}$ the evident trivializations, 
and let 
\[
\chi:\H^1_{\vert D\vert}(\gm,\K_1^{\MW},\OO_{\gm}(D))\to \H^1_{\vert D\vert}(\gm,\K_1^{\MW},\omega_{\gm/k})
\]
denote the induced isomorphism. 
For $p:\gm\to \Spec (k)$ and the induced push-forward map
\[
p_*:\H^1_{\vert D\vert}(\gm,\K_1^{\MW},\omega_{\gm/k})\to \K_0^{\MW}(k),
\]
we have $p_*\chi(D\cdot 1)=\langle -1\rangle$.
\end{lem}
\begin{proof}
Consider the sequence
\[
\xymatrix{\K_1^{\MW}(k(t),\omega_{k(t)/k})\ar[r]^-{d_0} & C_1(\mathbb{P}^1,\K_1^{\MW},\omega_{\mathbb{P}^1/k})\ar[r] & \K_0^{\MW}(k),}
\]
where the rightmost map is a sum of the (cohomological) transfer maps defined for instance in \cite[Definition 4.26]{Morel08}, 
see also \cite[\S 10.4]{Fasel08a}.
Since the push-forward map
\[
p_*:\H^1(\mathbb{P}^1,\K_1^{\MW},\omega_{\mathbb{P}^1/k})\to \K_0^{\MW}(k)
\]
is well defined \cite[Example 4.4]{Calmes14b}, 
it follows that the composition in the sequence is trivial, 
see also \cite[Chapitre 8]{Fasel08a}. 
Let $\overline D$ be the principal Cartier divisor on $\mathbb P^1$ defined by $g$. 
The evident trivialization $\OO_{\mathbb P^1}\to \OO_{\mathbb P^1}(\overline D)$ together with the canonical isomorphism
\[
\H^1(\mathbb{P}^1,\K_1^{\MW},\omega_{\mathbb{P}^1/k})\simeq \H^1(\mathbb{P}^1,\K_1^{\MW})
\]
yield an isomorphism 
$$
\overline \chi
\colon
\H^1_{\vert \overline D\vert}(\mathbb P^1,\K_1^{\MW},\OO_{\mathbb P^1}(\overline D))
\xrightarrow{\simeq}
\H^1_{\vert \overline D\vert}(\mathbb P^1,\K_1^{\MW},\omega_{\mathbb P^1/k}),
$$ 
and we have $p_*\overline\chi(\overline D\cdot 1)=0$. 
The restriction of the intersection product $\overline D\cdot 1$ to $\gm$ equals $D\cdot 1$, 
so that 
$$
\overline D\cdot 1=D\cdot 1+\widetilde{\mathrm{ord}}_\infty(\overline D)+\widetilde{\mathrm{ord}}_0(\overline D).
$$ 
It is straightforward to check that $\widetilde{\mathrm{ord}}_\infty(\overline D)=0$. 
Next we compute $\widetilde{\mathrm{ord}}_0(\overline D)$. 
Using the equality
$$
g=t^{-1}(t^{n+1}-1)/(t^n-1),
$$ 
we obtain
\begin{eqnarray*}
\partial_0([g]\otimes g^{-1}) & = & \partial_0([(t^{n+1}-1)/(t^n-1)]\otimes g^{-1}+\langle (t^{n+1}-1)/(t^n-1)\rangle [t^{-1}]\otimes g^{-1}) \\
 & = & \partial_0(\langle (t^{n+1}-1)/(t^n-1)\rangle [t^{-1}]\otimes g^{-1}) \\
 & = & \partial_0(\langle (t^{n+1}-1)/(t^n-1)\rangle\cdot \epsilon[t]\otimes g^{-1}) \\
 & = & \epsilon\otimes (t\otimes g^{-1}).
\end{eqnarray*}
Thus $\overline \chi(\overline D\cdot 1)=\chi(D\cdot 1)+\epsilon$, 
where $\epsilon$ is seen as an element of $\KMW_0(k)\subset C_1(\mathbb{P}^1,\K_1^{\MW},\omega_{\mathbb{P}^1/k})$. Finally, we have 
\[
0=p_*(\overline \chi(\overline D\cdot 1))=p_*(\chi(D\cdot 1)+\epsilon)=p_*(\chi(D\cdot 1))+\epsilon,
\]
and it follows that $p_*(\chi(D\cdot 1))=-\epsilon=\langle -1\rangle$.
\end{proof}


\section{Cancellation for Milnor-Witt correspondences}
\label{cancellationW}

If $X,Y\in\Sm_{k}$ we follow the convention in \cite[\S 4]{Voevodsky10} by letting $XY$ be short for the fiber product $X\times_{k} Y$.
We denote the dimension of $Y\in\Sm_{k}$ by $d_Y$.
Suppose $\alpha\in \tcor k(\gm X,\gm Y)$ is a finite Milnor-Witt correspondence. 
It has a well-defined support by \cite[Definition 4.6]{Calmes14b} which we denote by $T:=\mathrm{supp}(\alpha)$, 
so that 
\[
\alpha\in \H^{d_Y+1}_T(\gm X \gm Y ,\KMW_{d_Y+1},\omega_{\gm Y})
\]
where $\omega_{\gm Y}$ is the pull-back of the canonical sheaf of $\gm Y$ along the relevant projection.
Let $t_1$ and $t_2$ denote the (invertible) global sections on $\gm X \gm Y$ obtained by pulling back the coordinate on $\gm$ under the projections on the first and third factor of $\gm X\gm Y$,
respectively.
Recall from \cite[\S 4]{Voevodsky10} that there exists a natural number $N(\alpha)\in\N$ such that for all $n\geq N(\alpha)$ the rational function 
$$
g_n:=(t_1^{n+1}-1)/(t_1^{n+1}-t_2)
$$ 
defines a principal Cartier divisor $D(g_n)$ on $\gm X \gm Y$ having the property that $D(g_n)$ and $T$ intersect properly, 
and further that $\vert D(g_n)\vert \cap T$ 
is finite over $X$. 
If $n\geq N(\alpha)$, 
we say that $D(g_n)$ is \emph{defined relative to $\alpha$}. Unless otherwise specified, we always assume in the sequel that the Cartier divisors are defined relative to the considered cycles, 
e.g., 
for $D:=D(g_n)$ and $\alpha$.
 
By using the trivialization of $\OO(D)$, 
we see that 
$$
D\cdot \alpha \in \H^{d_Y+2}_{T\cap \vert D\vert} (\gm X \gm Y,\KMW_{d_Y+2},\omega_{\gm Y}).
$$ 
By permuting the first two factors in the product $\gm X\gm Y$ we obtain
$$
D\cdot \alpha\in \H^{d_Y+2}_{T\cap \vert D\vert} (X\gm\gm Y,\KMW_{d_Y+2},\omega_{\gm Y}).
$$ 
Similarly, 
by using the trivialization of the canonical line bundle $\omega_{\gm/k}$ and the canonical isomorphism 
\begin{equation}\label{eq:orientation}
\omega_{\gm\gm Y}\simeq p^*\omega_{\gm/k}\otimes \omega_{\gm Y}
\end{equation}
where $p$ is the projection onto the first factor, we obtain that  $D\cdot \alpha$ can be viewed as an element of $\tcor k(X,\gm\gm Y)$. 

\begin{defn}
Let $\rho_n(\alpha)\in \tcor k(X,Y)$ be the composite of $D\cdot \alpha\in\tcor k(X,\gm\gm Y)$ with the projection map $\gm\gm Y\to Y$. 
\end{defn}

\begin{lem}
\label{lem:minus1}
For $\alpha\in \tcor k(X,Y)$ and $1\times\alpha\in \tcor k(\gm X,\gm Y)$ we have
$$
\rho_n(1\times\alpha)=\alpha.
$$
\end{lem}
\begin{proof}
Let $\Delta:\gm\to \gm\gm$ be the diagonal embedding, 
and consider the commutative diagram of Cartesian squares in $\Sm_{k}$, 
where all the maps apart from $\Delta$ and $\Delta\times 1$ are projections:
\[
\xymatrix{
\gm XY \ar[r]^{p_2^\prime}\ar[d]_-{\Delta\times 1} & \gm\ar[d]^-{\Delta} \\ 
\gm\gm XY\ar[d]_-{p_1}\ar[r]^-{p_2} & \gm\gm\ar[d] \\
XY\ar[r]^-p & \Spec (k) 
}
\]
We have $1\times \alpha=p_2^*\Delta_*(\langle 1\rangle)\cdot p_1^*\alpha$ by definition, 
while the base change formula \cite[Proposition 3.2]{Calmes14b} imply $1\times\alpha=(\Delta\times 1)_*(\langle 1\rangle)\cdot p_1^*\alpha$.
On the other hand, 
let $D^\prime$ be the principal Cartier divisor on $\gm\gm$ defined by 
$$
g_n(t_1,t_2):=(t_1^{n+1}-1)/(t_1^{n+1}-t_2),
$$ 
and let $D^{\prime\prime}$ be the principal Cartier divisor on $\gm$ defined by 
$$
g_n(t):=(t^{n+1}-1)/(t^{n+1}-t).
$$ 
Since $\Delta^*(D^\prime)=D^{\prime\prime}$ and $p_2^*D^\prime=D$ we obtain $D\cdot (1\times\alpha)=p_2^*D^\prime\cdot (\Delta\times 1)_*(\langle 1\rangle)\cdot p_1^*\alpha$. Now $p_2^*D^\prime\cdot (\Delta\times 1)_*(\langle 1\rangle)=\langle -1\rangle (\Delta\times 1)_*(\langle 1\rangle)\cdot p_2^*D^\prime$. Using $(1\times\Delta)^*(p_2^*D^\prime)=(p_2^\prime)^*D^{\prime\prime}$ and the projection formula, 
we find 
$$
D\cdot (1\times\alpha)=\langle -1\rangle(1\times\Delta)_*((p_2^\prime)^*D^{\prime\prime})\cdot p_1^*\alpha.
$$ 
By composing this Milnor-Witt correspondence with $p_1$, 
which is tantamount to applying $(p_1)_*$ by \cite[Example 4.13]{Calmes14b}, 
we obtain
\[
\rho_n(1\times\alpha)
=
(p_1)_*(\langle -1\rangle(1\times\Delta)_*((p_2^\prime)^*D^{\prime\prime})\cdot p_1^*\alpha)
=
\langle -1\rangle(p^*q_*(D^{\prime\prime}\cdot \langle 1\rangle))\cdot \alpha,
\]
where $q:\gm\to \Spec(k)$. 
The result follows now from Lemma \ref{lem:degree}.
\end{proof}

\begin{lem}
For the Milnor-Witt correspondence 
\[
e_X:
\gm X
\xrightarrow{q}
X
\xrightarrow{\{1\}\times \mathrm{Id}}
\gm X,
\]
where $q$ is the projection map, 
we have $\rho_n(e_X)=0$ for any $n\in\N$.
\end{lem}
\begin{proof}
As in \cite[Lemma 4.3(ii)]{Voevodsky10}, 
note that the cycle representing the above composite is the image of the unit $\langle 1\rangle\in \KMW_0(\gm X)$ under the push-forward homomorphism
\[
f_*
\colon
\KMW_0(\gm X)
\to 
\H^{d_X+1}_{f(\gm X)}(\gm X\gm X,\KMW_{d_X+1}, \omega_{\gm X})
\]
for the map $f:\gm X\to \gm X\gm X$ given as the diagonal on $X$ and by $t\mapsto (t,1)$ on $\gm$. 
The result follows now from the projection formula since $f^*D(g_n)=1$ for all $n\in \N$. 
\end{proof}

\begin{lem}
\label{lem:composite}
Suppose $D:=D(g_n)$ is defined relative to $\alpha\in \tcor k(\gm X,\gm Y)$. 
Then for any Milnor-Witt correspondence $\beta:X^\prime\to X$, 
$D$ is defined relative to $\alpha\circ (1\times \beta)$, 
and we have
\[
\rho_n(\alpha\circ (1\times \beta))=\rho_n(\alpha)\circ \beta.
\]
\end{lem}

\begin{proof}
For the fact that $D$ is defined relative to $\alpha\circ (1\times \beta)$ we refer to \cite[Lemma 4.4]{Voevodsky10}. 
For the second assertion we note that $\rho_n(\alpha)\circ \beta$ is the composite
\[
X^\prime\stackrel{\beta}\to X\xrightarrow{D\cdot \alpha} \gm\gm Y\stackrel{q}\to Y,
\]
while $\rho_n(\alpha\circ (1\times \beta))$ is the composite
\[
X^\prime\xrightarrow{D\cdot(\alpha\circ (1\times \beta))} \gm\gm Y\stackrel{q}\to Y.
\]
Thus it suffices to prove that 
$$
D\cdot(\alpha\circ (1\times \beta))=(D\cdot \alpha)\circ \beta.
$$ 
Consider the following diagram
\[
\begin{gathered}
\xymatrix{
X^\prime\gm\gm Y \ar@/^2em/[rrrd]^-{s_Y}\ar@/_1.3em/[rddd]_-{s_X} & & & \\
 & X^\prime X\gm\gm Y\ar[r]^-{q_{XY}}\ar[d]_-{p_{X^\prime X}}\ar[lu]_-{\pi} & X\gm\gm Y\ar[r]^-{q_Y}\ar[d]^-{p_Y} & \gm\gm Y \\
 & X^\prime X\ar[r]_-{q_X}\ar[d]_-{p_{X^\prime}} & X & \\
 & X^\prime & & 
}
\end{gathered}
\]
By definition and functoriality of the pull-back, 
\[
(D\cdot \alpha)\circ \beta=\pi_*(p_{X^\prime X}^*\beta\cdot q_{XY}^*(D\cdot \alpha))=\pi_*(p_{X^\prime X}^*\beta\cdot D^\prime\cdot q_{XY}^*\alpha)
\]
where $D^\prime$ is the pull-back of $D$ along the projection $X^\prime X\gm\gm Y\to \gm\gm$. On the other hand, a direct computation involving the base change formula (\cite[Proposition 3.2]{Calmes14b}) shows that 
\[
\alpha\circ (1\times \beta)=\pi_*(p_{X^\prime X}^*\beta\cdot q_{XY}^*\alpha).
\]
Let then $D^{\prime\prime}$ be the pull-back of $D$ along the projection $X^\prime\gm\gm Y\to \gm\gm$. Using respectively the (skew-)commutativity of Chow-Witt groups, the projection formula and the commutativity again, we find
\begin{eqnarray*}
D^{\prime\prime}\cdot \pi_*(p_{X^\prime X}^*\beta\cdot q_{XY}^*\alpha) & = &\langle (-1)^{d_Y+1}\rangle \pi_*(p_{X^\prime X}^*\beta\cdot q_{XY}^*\alpha)\cdot  D^{\prime\prime} \\
 & = & \langle (-1)^{d_Y+1}\rangle \pi_*(p_{X^\prime X}^*\beta\cdot q_{XY}^*\alpha\cdot D^\prime) \\
 & = & \pi_*(p_{X^\prime X}^*\beta\cdot D^\prime \cdot q_{XY}^*\alpha).
\end{eqnarray*}
\end{proof}

\begin{lem}
\label{lem:rightcomposition}
Suppose $D:=D(g_n)$ is defined relative to $\alpha\in \tcor k(\gm X,\gm Y)$. 
Then for any map of schemes $f:X^\prime\to Y^\prime$, $D(g_n)$ is defined relative to $\alpha\times f$, 
and we have 
$$
\rho_n(\alpha\times f)=\rho_n(\alpha)\times f.
$$
\end{lem}
\begin{proof}
Omitting the required permutations, 
the product $\alpha\times f$ is defined by the intersection product of the pull-backs on $\gm X\gm Y X^\prime Y^\prime$ along the corresponding projections of $\alpha$ and the graph $\Gamma_f$. 
To compute $\rho_n(\alpha\times f)$ we multiply (on the left) by $\tildediv D$ (again omitting the permutations). 
On the other hand, 
the correspondence $\rho_n(\alpha)\times f$ is obtained by multiplying $\tildediv D$, $\alpha$, and $\Gamma_f$ along the relevant projections. 
Thus the assertion follows from associativity of the intersection product.
\end{proof}

Recall that $\tc\{1\}$ is the cokernel of the morphism $\tc(\Spec(k))\to \tc(\gm)$ induced by the unit for the multiplicative group scheme $\gm$.

\begin{lem}
\label{lem:hopf}
The twist map 
$$
\sigma:\tc(\gm)\otimes\tc(\gm)\to \tc(\gm)\otimes\tc(\gm)
$$
induces a map 
$$
\sigma:\tc\{1\}\otimes \tc\{1\}\to \tc\{1\}\otimes \tc\{1\},
$$ 
which is $\aone$-homotopic to multiplication by $\epsilon$.
\end{lem}

\begin{proof}
The claim that $\sigma$ induces a map $\tc\{1\}\otimes \tc\{1\}\to \tc\{1\}\otimes \tc\{1\}$ follows immediately from the definition.
By the proof of \cite[Lemma 4.8]{Voevodsky10}, 
$g\otimes f\colon X\to\gm{}\gm{}$ is $\aone$-homotopic to $f\otimes g^{-1}$ for all $X\in\Sm_{k}$ and $f,g\in \OO(X)^\times$.
We are thus reduced to proving that $\tc\{1\}\to\tc\{1\}$ induced by the map $\gm\to \gm$ sending $z$ to $z^{-1}$  is $\aone$-homotopic to multiplication by $\epsilon$. 
Its graph in $\gm\gm=\Spec(k[t_1^{\pm1},t_2^{\pm1}])$ is given by the prime ideal $(t_2-t_1^{-1})$, 
while the graph of the identity is given by $(t_2-t_1)$. 
We note that their product equals
$$
(t_2-t_1^{-1})(t_2-t_1)=t_2^2-(t_1+t_1^{-1})t_2+1. 
$$
On the other hand, we can consider $(t_2-1)^2$ and the polynomial 
$$
F:=t_2^2-t_2(u(t_1+t_1^{-1})+2(1-u))+1\in k[t_1^{\pm 1},t_2^{\pm 1},u].
$$ 
It is straightforward to check that $V(F)$ defines a closed subset in $\gm \A^1 \gm$ which is finite and surjective over $\gm \A^1$. 
The same properties hold for $V(G)$, 
where $G:=t_2^2-2t_2(1-2u)+1$. 

We define $F_1:=(t_1-t_1^{-1})F$, $F_2:=(t_1+1)G$, and consider the element 
\[
\alpha:=[F_1]\otimes dt_2 -\langle 2\rangle[F_2]\otimes dt_2-\langle 2\rangle[t_1-1]\otimes dt_2 \in \KMW_1(k(u,t_1,t_2),\omega_{\gm})
\] 
together with its image under the residue homomorphism
\[
d:\KMW_1(k(u,t_1,t_2),\omega_{\gm})\to \bigoplus_{x\in (\gm\A^1\gm)^{(1)}} \KMW_0(k(x),\omega_{\gm}).
\] 
of \cite[Remark 3.21]{Morel08}.
The polynomial $F_1$ ramifies at $F$ and at $(t_1-t_1^{-1})=t_1^{-1}(t_1-1)(t_1+1)$. 
The residue of $[F_1]\otimes dt_2$ at $(t_1-1)$ is 
$$
\langle 2\rangle \otimes (t_1-1)dt_2, 
$$
while its residue at $t_1+1$ is  
\[
\langle 2(t_2^2-2t_2(1-2u)+1)\rangle \otimes (t_1+1)dt_2.
\] 
On the other hand, 
the residue of $\langle 2\rangle[F_2]\otimes dt_2$ at $(t_1+1)$ is 
$$
\langle 2(t_2^2-2t_2(1-2u)+1)\rangle \otimes (t_1+1)dt_2.
$$ 
It follows that $\alpha$ is unramified at both $(t_1+1)$ and $(t_1-1)$. 
Hence the residue of $\alpha$ defines a Milnor-Witt correspondence $\gm\A^1 \to \gm$ of the form
\[
d(\alpha)\in \H^1_{V(F)}(\gm\A^1\gm,\KMW_1,\omega_{\gm})\subset \tcor k(\gm\A^1,\gm).
\]
We now compute the restriction of $d(\alpha)$ at $u=0$ and $u=1$, 
i.e., 
we compute the image of $d(\alpha)$ along the morphisms $\tcor k(\gm\A^1,\gm)\to \tcor k(\gm,\gm)$ induced by the inclusions $\Spec(k)\to \A^1$ at $u=0$ and $u=1$.
Its restriction $d(\alpha)(0)$ to $u=0$ is supported on $V(t_2-1)$ and is thus trivial on $\tc\{1\}(\gm)$. 
Its restriction $d(\alpha)(1)$ to $u=1$ takes the form
\[
\langle 1\rangle\otimes (t_2-t_1)dt_2+\langle -1\rangle \otimes (t_2-t_1^{-1})dt_2+ \langle 1,-1\rangle\otimes (t_2+1)dt_2.
\]
It follows that $d(\alpha)(1)$ is $\A^1$-homotopic to $0$. 
To conclude we are reduced to showing that $\langle 1,-1\rangle\otimes (t_2+1)dt_2$ is $\A^1$-homotopic to $0$. 
This is obtained using the polynomial $[G]$ by observing that the restriction to $u=0$ of the residue of $G$ is supported on $V(t_2-1)$, 
while the restriction to $u=1$ is precisely $\langle 1,-1\rangle\otimes (t_2+1)dt_2$.
\end{proof}


\begin{thm}
\label{thm:weakcancellation}
Let $\mathcal F$ be a Milnor-Witt presheaf and let $p:\tc(X)\to \mathcal F$ be an epimorphism of presheaves. 
Suppose 
\[
\phi: \tc\{1\}\otimes \mathcal F\to \tc\{1\}\otimes \tc(Y)
\]
is a morphism of presheaves with Milnor-Witt transfers. 
Then there exists a unique up to $\aone$-homotopy morphism $\psi:\mathcal F\to \tc(Y)$ such that $\phi\cong\mathrm{Id}\otimes \psi$.
\end{thm}
\begin{proof}
Precomposing $\phi$ with the projection $\tc(\gm)\otimes \mathcal F\to \tc\{1\}\otimes \mathcal F$ and postcomposing with the monomorphism $\tc\{1\}\otimes\tc(Y)\to \tc(\gm)\otimes \tc(Y)$, 
we get a morphism $\tilde \phi:\tc(\gm)\otimes \mathcal F\to \tc(\gm)\otimes \tc(Y)$. 
Composing $\tilde\phi$ with $\mathrm{Id}\otimes p$ yields a morphism of presheaves 
$$
\alpha: \tc(\gm)\otimes \tc(X)\to \tc(\gm)\otimes\tc(Y), 
$$ 
i.e., 
a Milnor-Witt correspondence $\alpha\in \tcor k(\gm X,\gm Y)$.
Choose $n\geq N(\alpha)$ such that $D(g_n)$ be defined relative to $\alpha$. 
Using Lemma \ref{lem:composite}, 
we get a morphism of presheaves $\rho_n(\tilde \phi):\mathcal F\to \tc(Y)$. 
If $\phi=\mathrm{Id}\otimes \psi$ for some $\psi:\mathcal F\to \tc(Y)$ we have $\alpha=\mathrm{Id}\otimes\psi p$. 
Lemma \ref{lem:minus1} shows that $\rho_n(\alpha)=\psi p$. 
Thus $\rho_n(\tilde \phi)=\psi$ and it follows that $\psi=\chi_n(\tilde \phi)$. 
This implies the uniqueness part of the theorem as in \cite[Theorem 4.6]{Voevodsky10}.

To conclude, 
we need to observe there is an $\aone$-homotopy $\mathrm{Id}\otimes \rho_n(\tilde \phi)\simeq \phi$. 
Let 
\[
\tilde \phi^*:\mathcal F\otimes \tc(\gm)\to \tc(Y)\otimes \tc(\gm)
\]
be the morphism obtained by permutation of the factors, 
and let 
\[
\phi^*:\mathcal F\otimes \tc\{1\}\to \tc(Y)\otimes \tc\{1\}
\]
be the naturally induced morphism. 
Using Lemma \ref{lem:hopf}, 
we see that $\mathrm{Id}_{\tc\{1\}}\otimes \phi^*$ and $\phi\otimes\mathrm{Id}_{\tc\{1\}}$ are $\A^1$-homotopic. 
It follows, 
using Lemmas \ref{lem:minus1}, \ref{lem:composite} and \ref{lem:rightcomposition}, 
that $\rho_n(\tilde\phi)\otimes \mathrm{Id}_{\tc\{1\}}=\phi^*$. 
That is, 
we have $\phi=\mathrm{Id}_{\tc\{1\}}\otimes \rho_n(\tilde\phi)$.
\end{proof}

Analogous to  \cite[Corollary 4.9]{Voevodsky10} we deduce from Theorem \ref{thm:weakcancellation} the following result for the Suslin complex of Milnor-Witt sheaves defined via the internal Hom object $\uHom$
and the standard cosimplicial scheme.

\begin{cor}\label{cor:cancel}
For any $Y\in\Sm_{k}$, the morphism
\[
\tc (Y)\to \uHom(\tc\{1\},\tc(Y)\otimes \tc\{1\})
\] 
induces for any smooth scheme $X$ a quasi-isomorphism of complexes
\[
\sus{\tc(Y)}(X)\xrightarrow{\sim} \sus{\uHom(\tc\{1\},\tc(Y)\otimes \tc\{1\})}(X).
\]
\end{cor}


\section{Zariski vs. Nisnevich hypercohomology}\label{ZvsN}

Throughout this short section we assume that $k$ is an infinite perfect field.
Let us first recall the following result from \cite[Theorem 3.2.9]{Deglise16}.

\begin{thm}
\label{thm:strictly}
Let $\mathcal F$ be an $\A^1$-invariant Milnor-Witt presheaf. 
Then the associated Milnor-Witt sheaf $\tilde a(\mathcal F)$ is strictly $\aone$-invariant. 
Moreover, 
the Zariski sheaf associated with $\mathcal F$ coincides with $\tilde a(\mathcal F)$, 
and for all $i\in \N$ and $x\in\Sm_{k}$ there is a natural isomorphism
\[
\H^i_{\Zar}(X,\tilde a(\mathcal F))
\xrightarrow{\simeq}
\H^i_{\Nis}(X,\tilde a(\mathcal F)).
\]
\end{thm}

We now derive some consequences of Theorem \ref{thm:strictly} following \cite[\S 13]{Mazza06}.

\begin{prop}\label{prop:preparatory}
Let $f\colon \mathscr C \to \mathscr D$ be a morphism of complexes of Milnor-Witt presheaves. 
Suppose their cohomology presheaves are homotopy invariant and that $\mathscr C (\Spec (F))\to \mathscr D (\Spec (F))$ is a quasi-isomorphism for every finitely generated field extension $F/k$. 
Then the induced morphism on the associated complexes of Zariski sheaves is a quasi-isomorphism. 
\end{prop}
\begin{proof}
Consider the mapping cone $\mathscr C_{f}$ of $f$. 
The five lemma implies that the cohomology presheaves $\H^i(\mathscr C_{f})$ are homotopy invariant and have Milnor-Witt transfers. 
The associated Nisnevich Milnor-Witt sheaves are strictly $\A^1$-invariant by Theorem \ref{thm:strictly}.
We have $\tilde a(\H^i(\mathscr C_{f}))(\Spec (F))=0$ by assumption. 
By \cite[Theorem 2.11]{Morel08} it follows that $\tilde a(\H^i(\mathscr C_{f}))=0$. 
This finishes the proof by applying Theorem \ref{thm:strictly}.
\end{proof}

By following the proof of \cite[Theorem 13.12]{Mazza06} with Theorem \ref{thm:strictly} and Proposition \ref{prop:preparatory} in lieu of the references in loc. cit., 
we deduce the following result.

\begin{thm}
Suppose $\mathcal F$ is a Milnor-Witt presheaf such that the Nisnevich sheaf $\tilde a(\mathcal F)=0$. 
Then $\tilde a(\sus {\mathcal F})$ is quasi-isomorphic to $0$. 
The same result holds for the complex of Zariski sheaves associated to $\sus {\mathcal F}$. 
\end{thm}

\begin{cor}\label{cor:localzero}
Let $f\colon \mathscr C\to \mathscr D$ be a morphism of bounded above complexes of Milnor-Witt presheaves. 
Suppose $f$ induces a quasi-isomorphism $f:\mathscr C(X)\to \mathscr D(X)$ for all Hensel local schemes $X$. 
Then $\mathrm{Tot}(\sus {\mathscr C})(Y)\to \mathrm{Tot}(\sus {\mathscr D})(Y)$ is a quasi-isomorphism for every regular local ring $Y$. 
\end{cor}
\begin{proof}
The proof of \cite[Corollary 13.14]{Mazza06} applies mutatis mutandis. 
\end{proof}

\begin{cor}\label{cor:local}
For every $X\in\Sm_{k}$ and regular local ring Z, the morphism of Milnor-Witt presheaves $\tc (X)\to \tZ(X)$ induces a quasi-isomorphism of complexes 
\[
\sus {\tc (X)}(Z)\xrightarrow{\simeq} \sus{\tZ(X)}(Z).
\]
There are naturally induced isomorphisms  
\[
\H^i_{\Zar}(Y,\sus {\tc (X)})\simeq \H^i_{\Zar}(Y,\sus{\tZ(X)})\simeq \H^i_{\Nis}(Y,\sus{\tZ(X)})
\]
for every $Y\in \Sm_{k}$.
\end{cor}
\begin{proof}
The isomorphism 
$$
\H^i_{\Zar}(Y,\sus {\tc (X)})
\simeq 
\H^i_{\Zar}(Y,\sus{\tZ(X)})
$$ 
follows immediately from Corollary \ref{cor:localzero}. 
The second isomorphism follows from \cite[Corollary 3.2.12]{Deglise16}.
\end{proof}

In the same vein, we have the following result.

\begin{cor}
\label{cor:local2}
For every $X\in\Sm_{k}$ and regular local ring $Z$, the morphism of Milnor-Witt presheaves 
$$
\tc (X)\otimes \tc\{1\}\to \tZ(X)\otimes \tZ\{1\}
$$ 
induces a quasi-isomorphism of complexes 
\[
\sus {\tc (X)\otimes \tc\{1\}}(Z)\xrightarrow{\simeq} \sus{\tZ(X)\otimes \tZ\{1\}}(Z).
\]
There are naturally induced isomorphisms
\[
\H^i_{\Zar}(Y,\sus {\tc (X)\otimes \tc\{1\}})\simeq \H^i_{\Zar}(Y,\sus{\tZ(X)\otimes \tZ\{1\}})\simeq \H^i_{\Nis}(Y,\sus{\tZ(X)\otimes \tZ\{1\}})
\]
for every $Y\in\Sm_{k}$.
\end{cor}
\begin{proof}
In the Nisnevich topology, 
$\tc (X)\otimes \tc\{1\}\to \tZ(X)\otimes \tZ\{1\}$ is locally an isomorphism and the proof of the previous corollary applies verbatim.
\end{proof}


\section{The embedding theorem for Milnor-Witt motives}\label{etfMWmotives}

\begin{thm}
\label{thm:main}
Let $k$ be an infinite perfect field. 
Then for all complexes $\mathscr C$ and $\mathscr D$ of Milnor-Witt sheaves, 
the morphism
\[
\Hom_{\DMt}(\mathscr C,\mathscr D)\to \Hom_{\DMt}(\mathscr C(1),\mathscr D(1))
\]
obtained by tensoring with the Tate object $\tZ(1)$ is an isomorphism.
\end{thm}
\begin{proof}
Obviously,
one can replace the Tate object $\tZ(1)$ by the $\gm$-twist $\tZ\{1\}$.
Using the internal Hom functor, one is reduced to proving there is a canonical isomorphism 
\[
\mathscr D \rightarrow \derR \uHom(\tZ\{1\},\mathscr D\{1\})
\]
for every Milnor-Witt motivic complex $\mathscr D$.
Since $\tZ\{1\}$ is a compact object in $\DMt$ by \cite[Proposition 3.2.21]{Deglise16} and $\mathscr D$ is a homotopy colimit of complexes obtained by suspensions of sheaves $\tZ(Y)$ for some smooth scheme $Y$, 
one is reduced to considering the case where $\mathscr D=\tZ(Y)$.
 
For every $n\in \Z$ and $X\in\Sm_{k}$, we have 
\[
\Hom_{\DMt}(\tZ(X),\tZ(Y)[n])=\H^n_{\Nis}(X,\sus{\tZ(Y)})
\]
by \cite[Corollaries 3.1.8, 3.2.14]{Deglise16}, while 
\[
\Hom_{\DMt}(\tZ(X),\derR \uHom(\tZ\{1\},\tZ(Y)\{1\})[n])=\Hom_{\DMt}(\tZ(X)\{1\},\tZ(Y)\{1\}[n])
\]
is the cokernel of the morphism
\[
\H^n_{\Nis}(X,\sus{\tZ(Y)\{1\}})\to \H^n_{\Nis}(X \gm,\sus{\tZ(Y)\{1\}})
\]
induced by the projection $X \gm\to X$. 
In view of Corollaries \ref{cor:local} and \ref{cor:local2}, we are reduced to proving that the tensor product by $\tZ\{1\}$ induces an isomorphism in Zariski hypercohomology. 
This follows from Corollary \ref{cor:cancel}.
\end{proof}

The category of Milnor-Witt motives $\DM$ is obtained from $\DMt$ by $\otimes$-inversion of the Tate object in the context of the stable model category structure on $\tZ(1)$-spectra \cite[\S 3.3]{Deglise16}.
By construction there is a "$\tZ(1)$-suspension" functor relating these two categories (\cite[Proposition 3.3.7]{Deglise16}).
From Theorem \ref{thm:main} we immediately deduce the analogue for Milnor-Witt motives of Voevodsky's embedding theorem.
\begin{cor}
\label{cor:embedding}
Let $k$ be an infinite perfect field.
There is a fully faithful suspension functor 
$$
\DMt
\to
\DM.
$$
\end{cor}


\section{Examples of $\aone$-homotopies}
\label{examples}

The purpose of this section is to perform a number of computations which will be useful both in \cite{Calmes14c} and \cite{Deglise16}. 
Let us first briefly recall the definition of the first motivic Hopf map in the context of Milnor-Witt correspondences. 
The computation in \cite[Lemma 3.6]{Calmes14b} shows that 
$$
\tcor k(\gm,\Spec (k))=\KMW_0(\gm)=\KMW_0(k)\oplus \KMW_{-1}(k)\cdot [t].
$$ 
Here,
the naturally defined class $[t]\in \KMW_{1}(k(t))$ is an element of the subgroup $\KMW_{1}(\gm)$ since it has trivial residues at all closed points of $\gm$.
Thus we have the element 
$$
\eta[t]\in\tcor k(\gm,\Spec (k))=\KMW_{0}(\gm).
$$
Under pull-back with the $k$-rational point $1\colon \Spec(k)\to\gm$ the element $[t]\in\KMW_{1}(\gm)$ goes to $[1]=0\in \KMW_{1}(k)$.
Thus $\eta[t]$ pulls back trivially to $\KMW_{0}(k)$, 
and defines a morphism of presheaves with MW-transfers $s(\eta)\in \mathrm{Hom}_{}(\tc \{1\},\tc(\Spec (k)))$. Now, $\tc(\Spec (k))=\KMW_0$ is a sheaf with MW-transfers, and it follows from \cite[Proposition 1.2.11]{Deglise16} that this morphism induces a morphism of sheaves with MW-transfers $\tZ\{1\}\to \tZ$ that we still denote by $s(\eta)$.

For any (essentially) smooth scheme $X\in\Sm_{k}$ the bigraded Milnor-Witt motivic cohomology group $\H_{\MW}^{n,i}(X,\Z)$ is defined using the complex $\tZ\{i\}$ \cite[Definition 3.3.6]{Deglise16}.
There is a graded ring structure on the Milnor-Witt motivic cohomology groups satisfying the commutativity rule in \cite[Theorem 3.4.3]{Deglise16}. 
Moreover, 
by \cite[\S 4.2.1]{Deglise16}, 
every unit $a\in F^{\times}$ (where $F$ is a finitely generated field exension of $k$) gives rise to an element of $\tcor k(\Spec (F),\gm)$ and also a Milnor-Witt motivic cohomology class $s(a)\in \H^{1,1}_{\mathrm{MW}}(F,\Z)$. Next, it follows from \cite[proof of Proposition 4.1.2]{Deglise16} that $s(\eta)$ yields a well-defined class in $\H_{\MW}^{-1,-1}(k,\Z)$ and thus a class in $\H_{\MW}^{-1,-1}(F,\Z)$ by pull-back. The definitions of $s(\eta)$ and $s(a)$ apply more generally to (essentially) smooth $k$-schemes \cite[\S4]{Deglise16}.

In analogy with motivic cohomology and Milnor $K$-theory, 
the integrally graded diagonal part of Milnor-Witt motivic cohomology can be identified with Milnor-Witt $K$-theory as sheaves of graded rings \cite[Theorem 4.2.3]{Deglise16}.
Using cancellation for Milnor-Witt correspondences we show results which are absolutely crucial for establishing the mentioned identification.

\begin{lem}
\label{lem:example1}
For every unit $a\in \OO(X)^\times$ we have 
$$
1+s(a)s(\eta)=\langle a\rangle\in \H^{0,0}_{\mathrm{MW}}(X,\Z)=\KMW_0(X).
$$
\end{lem}
\begin{proof}
For the function field $k(X)$ of $X$, 
the naturally induced map
$$
\H^{0,0}_{\mathrm{MW}}(X,\Z)
=
\KMW_0(X)
\to 
\KMW_0(k(X))
=
\H^{0,0}_{\mathrm{MW}}(k(X),\Z)
$$ 
is injective. 
We are thus reduced to proving the result for $X=\Spec(F)$, 
where $F$ is a finitely generated field extension of the base field $k$. 
Since this claim is obvious when $a=1$, we may assume that $a\neq 1$.

Consider the following diagram:
\[
\xymatrix{
\Spec(F) \gm \ar[r]^-{\Gamma_a\times 1}\ar[d]_-{p_1} & \Spec(F)\gm\gm\ar[r]^-{p_3}\ar[d]^-{p_2} & \gm \\
\Spec(F)\ar[r]^-{\Gamma_a} & \Spec(F) \gm & 
}
\]
Here,
$p_1,p_2$ and $p_3$ are the projections on the respective factors and $\Gamma_a$ is the graph of the morphism $a:\Spec(F)\to \gm$.
By construction $s(a)$ is the image of the cycle 
\[
(\Gamma_a)_{*}(\langle 1\rangle)\in \H ^1_{(t-a)}(F[t^{\pm 1}],\KMW_1,\omega_{F[t^{\pm 1}]/F})\subset \tcor k(F,\gm)
\] 
under the composite $\tcor k(F,\gm)\to \tc\{1\}(F)\to \tZ\{1\}(F)\to \H_{\MW}^{1,1}(F,\Z)$.
On the other hand, 
$s(\eta)$ corresponds to $\eta[t]\in \KMW_0(\gm)$ and the product is represented by the exterior product $\alpha$ of correspondences between $s(a)$ and $s(\eta)$.
The relevant Cartier divisor for the procedure described in \S\ref{cancellationW} is defined by the rational function
\[
g_n:=(t_1^{n+1}-1)/(t_1^{n+1}-t_2).
\]
In this case we may choose $N(\alpha)=1$. 
Since $\alpha:=p_2^*s(a)\cdot p_3^*s(\eta)$ and the intersection product is associative, setting $D:=D(g_0)$ we find that 
$$
D\cdot \alpha=(D\cdot p_2^*s(a))\cdot p_3^*s(\eta).
$$
Moreover,
using Remark \ref{rem:commutativity}, the base change and projection formulas, 
we get 
\[
D\cdot p_2^*s(a)=\langle -1\rangle p_2^*s(a)\cdot D= (\Gamma_a\times 1)_*(\langle -1\rangle)\cdot D=(\Gamma_a\times 1)_*(\langle -1\rangle(\Gamma_a\times 1)^*D).
\]
Now $D^\prime:=(\Gamma_a\times 1)^*D$ is the Cartier divisor on $\Spec (F) \gm$ defined by the rational function 
$$
g^\prime:=(t-1)/(t-a).
$$ 
Its divisor is of the form 
\[
\tildediv({D^\prime})=\langle 1-a\rangle\otimes (t-1)-\langle 1-a\rangle \otimes (t-a),
\]
and we have 
\[
(\Gamma_a\times 1)_*(\langle -1\rangle\tildediv({D^\prime}))=\langle a-1\rangle\otimes (t_1-1)\wedge (t_2-a)-\langle a-1\rangle \otimes (t_1-a)\wedge (t_2-a).
\]
Next, the projection formula yields
\[
(\Gamma_a\times 1)_*(\langle -1\rangle\tildediv({D^\prime}))\cdot p_3^*s(\eta)=(\Gamma_a\times 1)_*(\langle -1\rangle\tildediv({D^\prime})\cdot (\Gamma_a\times 1)^*p_3^*s(\eta))
\]
and $(\Gamma_a\times 1)^*p_3^*s(\eta)$ is just the pull-back of $s(\eta)$ along the projection $\Spec(F)\gm\to \gm$. Thus $D\cdot \alpha$ is the push-forward along $\Gamma_a\times 1$ of 
\[
\langle -1\rangle\tildediv({D^\prime})\cdot (\Gamma_a\times 1)^*p_3^*s(\eta)=\eta[1]\langle a-1\rangle\otimes (t_1-1)\wedge (t_2-a)-\eta[a]\langle a-1\rangle \otimes (t_1-a)\wedge (t_2-a).
\]
Since $[1]=0$ and $\eta[a]\langle a-1\rangle=\langle -1\rangle \eta[a]\langle 1-a\rangle=\langle -1\rangle\eta[a](1+\eta[1-a])=\langle -1\rangle\eta[a]$, 
we deduce that  
$$
\langle -1\rangle\tildediv({D^\prime})\cdot (\Gamma_a\times 1)^*p_3^*s(\eta)
=
-\langle -1\rangle\eta[a]
=
\eta[a].
$$
Pushing forward to $\Spec F$, we obtain $\rho_n(\alpha)=\eta[a]$. According to Theorem \ref{thm:weakcancellation}, it follows that $\alpha$ is $\A^1$-homotopic to the suspension of $\eta[a]=\langle a\rangle -1$.
\end{proof}

Let $\mu:\gm \gm\to \gm$ denote the multiplication map, and let $p_i:\gm \gm\to \gm$ for $i=1,2$ denote the projection map. 
We consider the corresponding graphs $\Gamma_\mu$, $\Gamma_i$, and their associated Milnor-Witt correspondences $\tilde\gamma_\mu$, $\tilde\gamma_i$. 
By construction we have $(\Gamma_\mu)_*(\langle 1\rangle)=\tilde\gamma_\mu$ and $(\Gamma_i)_*(\langle 1\rangle)=\tilde\gamma_i$. 
One checks that 
$$
\tilde\gamma_\mu-\tilde\gamma_1-\tilde\gamma_2\in \tcor k(\gm\gm,\gm)
$$ 
induces a morphism of Milnor-Witt presheaves $\alpha:\tc\{1\}\otimes \tc\{1\}\to \tc\{1\}$. 

\begin{lem}
\label{example2}
The morphism of Milnor-Witt presheaves $\alpha$ is $\A^1$-homotopic to the suspension of 
\[
s(\eta):\tc\{1\}\to \tc(\Spec (k)).
\]
\end{lem}
\begin{proof}
Let $t_1,t_2,t_3$ denote the respective coordinates of $\gm\gm \gm$ so that corresponding supports are given by $\supp{\tilde\gamma_\mu}=V(t_3-t_1t_2)$, $\supp{\tilde\gamma_1}=V(t_3-t_1)$, and $\supp{\tilde\gamma_2}=V(t_3-t_2)$. 
The Cartier divisors we want to employ are defined by the equation $g_n:=(t_2^{n+1}-1)/(t_2^{n+1}-t_3)$ for some $n\geq 0$.
We note that $D(g_n)$ intersects properly with both $\supp{\tilde\gamma_\mu}$ and $\supp{\tilde\gamma_1}$ if $n\geq 0$, 
while $D(g_n)$ and $\supp{\tilde\gamma_2}$ intersects properly if $n\geq 1$.
For this reason we set $D:=D(g_1)$. 

Next we compute the intersection product $D\cdot (\tilde\gamma_\mu-\tilde\gamma_1-\tilde\gamma_2)$. 
By the projection formula, 
we get 
\[
D\cdot \tilde\gamma_\mu=\langle -1\rangle(\Gamma_\mu)_*(\langle 1\rangle)\cdot D=(\Gamma_\mu)_*(\Gamma_\mu^*D\cdot \langle -1\rangle).
\]
Note that $D_\mu:=\Gamma_\mu^*D$ is the principal Cartier divisor associated to the rational function
\[
g_\mu:=(t_2^2-1)/(t_2^2-t_1t_2)=t_2^{-1}(t_2^2-1)/(t_2-t_1).
\]
Its associated divisor is given explicitly by
\[
\tildediv(D_\mu)=\langle 2\rangle\langle 1-t_1\rangle\otimes (t_2-1)+\langle -2\rangle \langle 1+t_1\rangle\otimes (t_2+1)-\langle t_1\rangle\langle 1-t_1^2\rangle\otimes (t_2-t_1).
\]
Similarly for $\tilde\gamma_i$ we get Cartier divisors $D_i:=\Gamma_i^*D$ for which 
\[
\tildediv(D_1)=\langle 2\rangle \langle 1-t_1\rangle\otimes (t_2-1)+\langle -2\rangle \langle 1-t_1\rangle\otimes (t_2+1)-\langle 1-t_1\rangle\otimes (t_2^2-t_1),
\]
and 
\[
\tildediv(D_2)=\langle 1\rangle \otimes (t_2+1).
\]
Pushing forward along $\Gamma_\mu$, $\Gamma_1$, $\Gamma_2$, 
and cancelling terms, 
we find an expression for $D\cdot (\tilde\gamma_\mu-\tilde\gamma_1-\tilde\gamma_2)$ of the form $\langle -1\rangle\beta dt_3$, 
where $\beta$ equals
\[
-\langle t_1(1-t_1^2)\rangle\otimes (t_2-t_1)\wedge (t_3-t_2^2)+\langle 1-t_1\rangle\otimes (t_2^2-t_1)\wedge (t_3-t_2^2)-\langle 1\rangle \otimes (t_2+1)\wedge (t_3-1).
\]
Next we need to modify the orientation following (\ref{eq:orientation}), 
see \cite[\S2.2, Example 4.14]{Calmes14b} for the exact meaning,
which in our case amounts to multiplying by $dt_2$ on the left. 
This yields the expression $\beta dt_2\wedge dt_3$ which takes the form
\[
-\langle t_1(1-t_1^2)\rangle +\langle 1-t_2^2\rangle\langle 2t_2\rangle -\langle 1\rangle.
\]
Here, 
the first factor belongs to $\KMW_0(k(x_1))$ for $x_1=(t_2-t_1,t_3-t_2^2)$, 
the second factor to $\KMW_0(k(x_2))$ for $x_2=(t_2^2-t_1,t_3-t_2^2)$, 
and the third factor to $\KMW_0(k(x_3))$ for $x_3=(t_2+1,t_3-1)$. 
In the next step we take the push-forward to 
$$
\KMW_0(\gm)=\KMW_0(k[t_1^{\pm 1}])\subset \KMW_0(k(t_1)).
$$ 
Since $k(x_1)=k(t_1)$ and $k(x_3)=k(t_1)$, the push-forwards of the first and third factors are evident. 
Note that $k(x_2)/k(t_1)$ is a degree $2$ field extension of the form $k(x_2)=k(t_2)=k(t_1)[t_2]/(t_2^2-t_1)$.
Thus we need to push-forward the form $\langle 1-t_2^2\rangle\langle 2t_2\rangle$ using the trace map, 
see \cite[Lemma 6.8]{Calmes14b}. 
By the projection formula, the result is the product of $\langle 1-t_2^2\rangle=\langle 1-t_1\rangle$ with the push-forward of $\langle 2t_2\rangle$. 
An easy computation shows the push-forward of $\langle 2t_2\rangle$ is hyperbolic, 
and we find that the push-forward of $\beta dt_2\wedge dt_3$ along the projection $\gm\gm\gm\to \gm$ to the first factor equals 
\[
\nu:=-\langle t_1(1-t_1^2)\rangle+\langle 1,-1\rangle -\langle 1\rangle=\langle -1\rangle -\langle t_1(1-t_1^2)\rangle.
\]
Finally, 
we need to isolate the "pointed" component of $\nu\in \tcor k(\gm,\Spec(k))$, 
i.e., 
to compute its image along the projection 
$$
\tcor k(\gm,\Spec(k))\to \tc\{1\}(\Spec(k)).
$$ 
This is obtained via the connecting homomorphism
\[
\KMW_0(\gm)\to \H_{\{0\}}^1(\A^1,\KMW_0)=\KMW_{-1}(k),
\]
under which $\nu\mapsto -\eta[t_1]=\langle -1\rangle \eta[t_1]$. 
Using Theorem \ref{thm:weakcancellation}, we can now conclude that $\alpha$ is $\A^1$-homotopic to $s(\eta)$. 
\end{proof}

\section{Milnor-Witt motivic cohomology}
\label{MWmcs}

In this section we lay the foundations for Milnor-Witt motivic cohomology from the perspective of motivic functors and structured motivic ring spectra following the work on motivic cohomology in \cite{DRO}.
This is of interest in the context of the very effective slice filtration \cite[\S5]{SO} due to recent work on hermitian $K$-theory in \cite{2016arXiv161001346B}.

A motivic space with Milnor-Witt tranfers is an additive contravariant functor
$$
\mathcal{F}
\colon
\tcor k
\to
\sAb
$$
from the category of Milnor-Witt correspondences to simplicial abelian groups.
Let $\tm_k$ denote the functor category comprised of motivic spaces with Milnor-Witt tranfers. 
By composition with the opposite of the graph functor $\Sm_{k}\to\tcor k$ (described explicitly in \cite[\S 4.3]{Calmes14b}) we obtain the forgetful functor 
$$
u\colon\tm_k\to \m_k
$$ 
to motivic spaces over $k$,
i.e., 
the category of simplicial presheaves on the Nisnevich site of $\Sm_{k}$.
By adding Milnor-Witt transfers to motivic spaces we obtain a left adjoint functor of $u$, 
denoted by
$$
\tZ_{\text{tr}}
\colon
\m_k
\to 
\tm_k.
$$
More precisely, 
$\tZ_{\text{tr}}$ is the left Kan extension determined by setting
$$
\tZ_{\text{tr}}(h_{X}\wedge \Delta^{n}_{+})
:=
\tcor k(-,X)\otimes\mathbb{Z}[\Delta^{n}].
$$
Here,
$h_{X}$ denotes the motivic space represented by $X\in\Sm_{k}$.
Recall that a motivic functor is an $\m_k$-enriched functor from finitely presented motivic spaces $\f\m_{k}$ to $\m_{k}$ \cite[Definition 3.1]{DRO}.
We write $\mathbf{MF}_{k}$ for the closed symmetric monoidal category of motivic functors with unit the full embedding $\f\m_{k}\subset\m_{k}$.
A motivic functor $\mathcal{X}$  is "continuous" in the sense that it induces for all $A,B\in\f\m_{k}$ a map of internal hom objects
$$
\m(A,B)
\to
\m(\mathcal{X}(A),\mathcal{X}(B)),
$$
which is compatible with the enriched composition and identities.

\begin{defn}
As a motivic functor, 
Milnor-Witt motivic cohomology $\mathbf{M}\widetilde{\mathbb{Z}}$ is the composite
$$
\f\m_{k}
\xrightarrow{\subset}
\m_{k}
\xrightarrow{\tZ_{\text{tr}}}
\tm_{k}
\xrightarrow{u}
\m_{k}.
$$
\end{defn}
We note that $\mathbf{M}\widetilde{\mathbb{Z}}$ is indeed a motivic functor since for all $A,B\in\f\m_{k}$ there exist natural maps 
$$
\m(A,B)\wedge u\tZ_{\text{tr}}(A)
\to
u\tZ_{\text{tr}}\m(A,B)\wedge u\tZ_{\text{tr}}(A)
\to
u\tZ_{\text{tr}}(\m(A,B)\wedge A)
\to
u\tZ_{\text{tr}}(B).
$$

\begin{lem}
\label{lem:MZcommutativemonoid}
Milnor-Witt motivic cohomology $\mathbf{M}\widetilde{\mathbb{Z}}$ is a commutative monoid in $\mathbf{MF}_{k}$.
\end{lem}
\begin{proof}
We note that $\tm_{k}$ is a closed symmetric monoidal category using \cite[\S1.2.13]{Deglise16}.
Clearly the graph functor $\Sm_{k}\to\tcor k$ is strict symmetric monoidal. 
Forgetting the additive structure furnished by Milnor-Witt correspondences is a lax symmetric monoidal functor.
By adjointness it follows that $u$ is lax symmetric monoidal, 
and $\tZ_{\text{tr}}$ is strict symmetric monoidal.
This yields the desired multiplicative structure on $\mathbf{M}\widetilde{\mathbb{Z}}$ via the natural maps
$$
u\tZ_{\text{tr}}(A)\wedge u\tZ_{\text{tr}}(B)
\to
u(\tZ_{\text{tr}}(A)\otimes\tZ_{\text{tr}}(B))
\to
u(\tZ_{\text{tr}}(A\wedge B))
$$
defined for all $A,B\in\f\m_{k}$.
\end{proof}
\begin{rem}
Lemma \ref{lem:MZcommutativemonoid} and \cite[Theorem 4.2]{DRO} show the category of modules over $\mathbf{M}\widetilde{\mathbb{Z}}$ is a cofibrantly generated monoidal model category satisfying the monoid axiom.
\end{rem}

If $\mathcal{X}$ is a motivic functor the $\m$-enrichment yields an induced map 
$$
A\wedge \mathcal{X}(B)\to\mathcal{X}(A\wedge B)
$$ 
for all finitely presented $A,B\in\f\m_{k}$.
In particular, 
for the Thom space $T:=\aone/\aone-\{0\}$ of the trivial line bundle on $\aone$, 
there are maps of motivic spaces $T\wedge \mathcal{X}(T^{n})\to\mathcal{X}(T^{n+1})$ for $n\geq 0$.
By \cite[\S3.7]{DRO} this yields a lax symmetric monoidal evaluation functor from $\mathbf{MF}_{k}$ to motivic symmetric spectra $\mathrm{Sp}^{\Sigma}(\m_{k},T)$ introduced in \cite{JardineMSS}.
Due to the Quillen equivalence between $\mathbf{MF}_{k}$ and $\mathrm{Sp}^{\Sigma}(\m_{k},T)$ \cite[Theorem 3.32]{DRO} we also write $\mathbf{M}\widetilde{\mathbb{Z}}$ for the corresponding 
motivic symmetric spectrum of Milnor-Witt motivic cohomology.
Equivalently, 
following the approach in \cite[\S4.2]{DRO},
we may view $\mathbf{M}\widetilde{\mathbb{Z}}$ as a motivic symmetric spectrum with respect to the suspension coordinate $(\mathbb{P}^{1},\infty)$.
Recall that a motivic symmetric spectrum $\mathcal{E}$ is fibrant if and only if it is an $\Omega_{T}$-spectrum, 
i.e., 
$\mathcal{E}_{n}$ is motivic fibrant and $\mathcal{E}_{n}\to \Omega_{T}\mathcal{E}_{n+1}$ is a motivic weak equivalence for all $n\geq 0$.
Voevodsky's cancellation theorem for finite correspondences implies the motivic cohomology spectrum is fibrant \cite[Theorem 6.2]{VoevodskyICM1998}.
In the context of Bredon motivic cohomology and the cancellation theorem for equivariant finite correspondences \cite[Theorem 9.7]{MR3398714}, 
this is shown in \cite[Theorem 3.4, Appendix A.4]{2016arXiv160207500H}.
By the same type of arguments we obtain:.
\begin{thm}
\label{MWfibrant}
If $k$ is an infinite perfect field then the Milnor-Witt motivic cohomology spectrum $\mathbf{M}\widetilde{\mathbb{Z}}$ is an $\Omega_{T}$-spectrum in $\mathrm{Sp}^{\Sigma}(\m_{k},T)$.
\end{thm}

\begin{rem}
\label{rem:MWrepresentability}
Along the lines of \cite[Theorem 3.4]{2016arXiv160207500H} one can show that $\mathbf{M}\widetilde{\mathbb{Z}}$ represents Milnor-Witt motivic cohomology groups as defined in $\DMt$
\cite[Definition 3.3.6]{Deglise16}.
\end{rem}

Due to work of Morel \cite{Morel04b} the integral graded homotopy module of the unramified Milnor-Witt $K$-theory sheaf corresponds to a motivic spectrum $\mathbf{M}^{\heartsuit}\widetilde{\mathbb{Z}}$ 
in the heart $\SH^{\heartsuit}(k)$ of the homotopy $t$-structure on the stable motivic homotopy category $\SH(k)$.
Here the equivalence of categories between homotopy modules, 
i.e., 
sequences $\{\mathcal{F}_{n}\}_{n\in\Z}$ of strictly $\aone$-homotopy invariant sheaves of abelian groups together with contraction isomorphisms $\mathcal{F}_{n}\cong (\mathcal{F}_{n+1})_{(-1)}$ for all $n\geq 0$,
and the heart of the stable motivic homotopy category is induced by the stable $\aone$-homotopy sheaf $\underline{\pi}_{0,\ast}$ with inverse $\mathbf{M}^{\heartsuit}$.
Recall that for every $\mathcal{E}\in\SH(k)$ the Nisnevich sheaf $\underline{\pi}_{s,t}(\mathcal{E})$ on $\Sm_{k}$ is obtained from the presheaf
$$
X
\mapsto
\SH(k)(\Sigma^{s,t}X_{+},\mathcal{E}), 
\text{ where } s,t\in\Z.
$$
By using basic properties of effective homotopy modules, 
Bachmann \cite[Lemma 12]{2016arXiv161001346B} shows there is an isomorphism of Nisnevich sheaves
$$
\underline{\pi}_{\ast,\ast}\mathbf{M}^{\heartsuit}\widetilde{\mathbb{Z}}
\cong
\K_{\ast}^{\MW}.
$$

\begin{lem}
There is a canonically induced isomorphism of Nisnevich sheaves
$$
\K_{\ast}^{\MW}
\xrightarrow{\cong}
\underline{\pi}_{\ast,\ast}\mathbf{M}\widetilde{\mathbb{Z}}.
$$
\end{lem}
\begin{proof}
This follows from \cite[Theorem 4.2.2]{Deglise16} together with the representability of Milnor-Witt motivic cohomology stated in Remark \ref{rem:MWrepresentability}.
\end{proof}

It is very plausible that $\mathbf{M}^{\heartsuit}\widetilde{\mathbb{Z}}$ and $\mathbf{M}\widetilde{\mathbb{Z}}$ are isomorphic, 
as noted in \cite[\S1]{2016arXiv161001346B}. 
Having such an isomorphism would be of interest in the context of the very effective slice filtration \cite[\S5]{SO} due to the work on hermitian $K$-theory $\KQ$ in \cite{2016arXiv161001346B}.
Combined with \cite[Theorem 10]{2016arXiv161001346B} and the fact that the unit map $\1\to\KQ$ becomes an isomorphism after applying the zeroth effective slice functor $\widetilde{s}_{0}$, 
we would be able to conclude there is a distinguished triangle 
$$
\Sigma^{1,0}\mathbf{M}\mathbb{Z}/2
\to
\widetilde{s}_{0}\1
\to
\mathbf{M}\widetilde{\mathbb{Z}}
$$
in $\SH(k)$, 
expressing the zeroth effective slice of the motivic sphere spectrum $\mathbf{1}$ as an extension of mod $2$ motivic cohomology $\mathbf{M}\mathbb{Z}/2$ and integral Milnor-Witt motivic cohomology 
$\mathbf{M}\widetilde{\mathbb{Z}}$.
This would provide a powerful computational tool since every effective slice is a module over $\widetilde{s}_{0}\1$ according to \cite[\S3.3, \S6(iv),(v)]{GRSO}.



\begin{footnotesize}
\bibliographystyle{alpha}
\bibliography{General}
\end{footnotesize}
\end{document}